\documentclass[preprint,12pt]{elsarticle}

\usepackage{amsfonts,amsthm,amsmath}
\usepackage[dvips]{epsfig}
\usepackage[dvips]{graphicx}
\usepackage{tikz}

\theoremstyle{plain}
\numberwithin{equation}{section}
\newtheorem{theorem}[equation]{Theorem}
\newtheorem{lemma}[equation]{Lemma}

\newtheorem{example}[equation]{Example}

\newtheorem{conjecture}[equation]{Conjecture}

\begin{document}

\begin{frontmatter}

\title{The most unbalanced words $0^{q-p}1^p$ and majorization}

\author{Jetro Vesti}
\ead{jejove@utu.fi}

\address{University of Turku, Department of Mathematics and Statistics, 20014 Turku, Finland}

\begin{abstract}
A finite word $w\in\{0,1\}^*$ is \emph{balanced} if for every equal-length factors $u$ and $v$ of every cyclic shift of $w$ we have $||u|_1-|v|_1|\leq 1$.
This new class of finite words were defined in \cite{jz}.

In \cite{j}, there was proved several results considering finite balanced words and majorization.
One of the main results was that the base-2 orbit of the balanced word is the least element
in the set of orbits with respect to partial sum. It was also proved that the product of the
elements in the base-2 orbit of a word is maximized precisely when the word is balanced.

It turns out that the words $0^{q-p}1^p$ have similar extremal properties, opposite to the balanced words,
which makes it meaningful to call these words \emph{the most unbalanced words}.
This article contains the counterparts of the results mentioned above.
We will prove that the orbit of the word $u=0^{q-p}1^p$ is the greatest element in the set of orbits with respect to partial sum
and that it has the smallest product. We will also prove that $u$ is the greatest element in the set of orbits with respect to partial product.
\end{abstract}

\journal{Discrete Math., Alg. and Appl. $\copyright$ [2015] [World Scientific]}

\end{frontmatter}


\section{Introduction}

\emph{Sturmian words} were first studied by Morse and Hedlund in \cite{mh}, since then being one of the core interests in combinatorics on words,
and \emph{finite balanced words} were first introduced and studied in \cite{jz}.
These words are closely linked because every Sturmian word is \emph{balanced}, as an infinite word, and every finite balanced word is a factor of some Sturmian word.
Every factor of a Sturmian word is not however necessarily a finite balanced word.
For a well-known survey on Sturmian words one should look \cite{bs}.
In this article we will study only finite words. 

Besides \emph{balanced words}, this paper is related more closely to \emph{the most unbalanced words} $0^{q-p}1^p$.
In \cite{j} there were given many new properties for balanced words in terms of majorization.
Majorization is a common notion in many branches of mathematics and has many applications, for example in probability, statistics and graph theory.
We will notice that also the words $0^p1^{q-p}$ have many extremal properties in terms of majorization, opposite to the balanced words.
This makes it meaningful to call these words the most unbalanced words.

In section 2 we will give the counterpart of Theorem 2.3 from \cite{j},
which says that the base-2 orbit of the balanced word is the least element in the set of orbits with respect to partial sum.
In this article we will prove that the base-2 orbit of the most unbalanced word is the greatest element.
Hence, this result places every other word, with the same number of ones and zeros, between these two extremal words.

In section 3 we will give the counterpart of Theorem 1.2 from \cite{j},
which says that the product of the elements in the base-2
orbit of a word is maximized precisely when the word is balanced.
In this article we will prove that the product is minimized when the word is the most unbalanced word.
This is done in the case where the number of zeros is greater than the number of ones.

In section 4 we will do the same for partial product what we did for partial sum in section 2.
The result that the base-2 orbit of the balanced word is the least element also with respect to partial product
has not been proved, but it seems very likely to be true. In any case, we will prove that the base-2 orbit of the most
unbalanced word is the greatest element also with respect to partial product.
For this, we will use the result from section 3, which means that the result is
proved in the case where the number of zeros is greater than the number of ones.

\subsection{Definitions and notation}

We denote by $A$ an \emph{alphabet}, i.e. a non-empty finite set of symbols called \emph{letters}.
A \emph{word} $w$ over $A$ is a finite sequence $w=w_1 w_2 \ldots w_n$, where $\forall i: w_i\in A$.
The \emph{empty word} $\epsilon$ is the empty sequence.
The set $A^*$ of all words over $A$ is a free monoid under the operation of concatenation
with identity element $\epsilon$ and set of generators $A$.
The free semigroup $A^+=A^*\setminus \{\epsilon\}$ is the set of non-empty words over $A$.

The \emph{length} of a word $w=w_1 w_2 \ldots w_n\in A^n$ is denoted by $|w|=n$.
The empty word is the unique word of length $0$.
By $|w|_a$, where $a\in A$, we denote the number of occurrences of the letter $a$ in $w$.
A word $x$ is a \emph{factor} of a word $w\in A^*$ if $w=uxv$, for some $u,v\in A^*$.
If $u=\epsilon$ ($v=\epsilon$) then we say that $x$ is a \emph{prefix} (resp. \emph{suffix}) of $w$.
The set $\textrm{F}(w)$ is the set of all factors of $w$
and the set $\textrm{Alph}(w)$ is the set of all letters that occur in $w$.

Other basic definitions and notation in combinatorics on words can be found from Lothaires books \cite{l1} and \cite{l2}. 

\subsection{Preliminaries}

The \emph{lexicographic order} of words $u=u_1\ldots u_n$ and $v=v_1\ldots v_n$ in $A^n$ is defined by:
$u<v$ if there exists $j\in\{1,\ldots,n\}$ such that $u_k=v_k$ for all $k=1,\ldots,j-1$ and $u_j<v_j$.
We denote $u\leq v$ if either $u<v$ or $u=v$.
The \emph{cyclic shift} $\sigma: A^n\to A^n$ is defined by $\sigma(w_1\ldots w_n)=w_2\ldots w_n w_1$.
The \emph{orbit} $\mathcal{O}(w)$ of a word $w\in A^n$ is the vector
$$\mathcal{O}(w)=(\mathcal{O}_1(w),\ldots,\mathcal{O}_n(w)),$$
where the words $\mathcal{O}_i(w)$ are the iterated cyclic shifts $w,\sigma(w),\ldots,\sigma^{n-1}(w)$
arranged in lexicographic order from the smallest to the largest.

In this paper we will restrict ourselves to binary alphabet, i.e. $A=\{0,1\}$, where we define that $0<1$.
We will denote by $(w)_2=\sum_{i=1}^{n}{w_i 2^{n-i}}$ the base-2 expansion of a word $w=w_1w_2\ldots w_n$
and define the \emph{base-2 orbit} of $w$ by
$$\mathcal{I}(w)=(\mathcal{I}_1(w),\ldots,\mathcal{I}_n(w))=((\mathcal{O}_1(w))_2,\ldots,(\mathcal{O}_n(w))_2).$$

A finite word $w\in\{0,1\}^*$ is \emph{balanced} if for every equal-length factors $u$ and $v$ of every cyclic shift of $w$ we have $||u|_1-|v|_1|\leq 1$.
If a word is not balanced then it is \emph{unbalanced}.
Notice that for example the word $001010$ is not balanced even though it is a factor of a Sturmian word and hence a factor of an \emph{infinite} balanced word.

Let $p$ and $q$ be coprime integers such that $1\leq p < q$.
$\mathcal{W}_{p,q}$ will denote the set of binary words $w\in A^q$ such that $|w|_1=p$ and $|w|_0=q-p$.
From \cite{bs} we know that there are $q$ balanced words in $\mathcal{W}_{p,q}$ and they are all in the same orbit.
We define $\mathbb{W}_{p,q}$ to be the set of all orbits in $\mathcal{W}_{p,q}$ and get that there is a unique balanced orbit in each $\mathbb{W}_{p,q}$.
Because the orbit depends only from one of its components, we will use the lexicographically smallest component $\mathcal{O}_1(w)$ to represent the orbit.
So for example the orbit $(00101,01001,01010,10010,10100)$ will be represented notationally by $00101$.

\begin{example}
If $(p,q)=(2,5)$ then the set of all orbits is $\mathbb{W}_{2,5}=\{00011,00101\}$,
where $00011=(00011,00110,01100,10001,11000)$, $00101=(00101,01001,01010,10010,10100)$.
The base-2 orbits are $\mathcal{I}(00011)=(3,6,12,17,24)$ and $\mathcal{I}(00101)=(5,9,10,18,20)$.
\end{example}

For $w,w'\in\mathbb{W}_{p,q}$ the base-2 orbit $\mathcal{I}(w)=(w_1,\ldots,w_q)$ of $w$ is said to
\emph{majorize} the base-2 orbit $\mathcal{I}(w')=(w'_1,\ldots,w'_q)$ of $w'$, denoted $w' \prec w$, if
$$\sum^{i}_{k=1}{w'_k} \geq \sum^{i}_{k=1}{w_k}\quad {\rm for}\ 1\leq i\leq q.$$
The majorization defines a partial order on the set $\mathbb{W}_{p,q}$.
We can easily calculate that $\sum^{q}_{k=1}{w'_k} = \sum^{q}_{k=1}{w_k} = (2^{q}-1)p$, which was stated already in \cite{jz} after Definition 2.1.
We denote the partial sums of the orbit of $w$ by $\mathcal{S}_i(w)=\sum^{i}_{k=1}{\mathcal{I}_k(w)}$.

Similarly, for $w,w'\in\mathbb{W}_{p,q}$ the base-2 orbit of $w$ is said to
\emph{majorize with respect to product} the base-2 orbit of $w'$, denoted $w' \prec_p w$, if
$$\prod^{i}_{k=1}{w'_k} \geq \prod^{i}_{k=1}{w_k}\quad {\rm for}\ 1\leq i\leq q.$$
The majorization with respect to product also defines a partial order on the set $\mathbb{W}_{p,q}$.
We denote the partial products of the orbit of $w$ by $\mathcal{P}_i(w)=\prod^{i}_{k=1}{\mathcal{I}_k(w)}$.

Let us then present Jenkinsons theorems from \cite{j}.

\begin{theorem}\label{t1}(\cite{j}, Thm. 2.3)
For any coprime integers $1\leq p < q$, the unique balanced orbit $b\in\mathbb{W}_{p,q}$
is the least element in $(\mathbb{W}_{p,q},\prec)$. In other words, for any $w\in\mathbb{W}_{p,q}$,
$$ \mathcal{S}_i(b) \geq \mathcal{S}_i(w)\quad {\rm for\ all}\ 1\leq i\leq q.$$
\end{theorem}

\begin{theorem}\label{t2}(\cite{j}, Thm. 1.2)
Suppose $1\leq p < q$ are coprime integers. For $w\in \mathbb{W}_{p,q}$ the product
$P(w)=\prod^{q}_{i=1}{\mathcal{I}_i(w)}$ is maximized precisely when $w$ is balanced.
\end{theorem}

The next conjecture, stating Theorem \ref{t1} for partial product, is very likely to be true.
One might be able to prove it using similar idea that Jenkinson used for partial sum.
In any case, we will prove the counterpart of it for the most unbalanced word.

\begin{conjecture}\label{c3}
For any coprime integers $1\leq p < q$, the unique balanced orbit $b\in\mathbb{W}_{p,q}$
is the least element in $(\mathbb{W}_{p,q},\prec_p)$. In other words, for any $w\in\mathbb{W}_{p,q}$,
$$ \mathcal{P}_i(b) \geq \mathcal{P}_i(w)\quad {\rm for\ all}\ 1\leq i\leq q.$$
\end{conjecture}


\section{Partial sum}

In this section we will prove the counterpart for Theorem \ref{t1}. We will not need the condition that $q$ and $p$ are coprime.
The word $u=0^{q-p}1^p$ is called \emph{the most unbalanced word} in $\mathcal{W}_{p,q}$ and the orbit
of $u$ is called \emph{the most unbalanced orbit} in $\mathbb{W}_{p,q}$.

\begin{example}
If $(p,q)=(3,8)$ then the set of all orbits is $\mathbb{W}_{3,8}=\{00000111,00001011,$ $00001101,00010011,00010101,00011001,00100101\}$.
The base-2 orbits and the partial sums of those orbits are listed in Table \ref{table:tableSum}.
From these partial sums we can see the partial ordering of the set $\mathbb{W}_{3,8}$, which is drawn in Figure \ref{pic1}.
\end{example}

\begin{table}[ht]
\begin{center}
\begin{tabular}{|c|c|c|c|c|c|c|c|c|c|c|c|c|c|}
\hline
\multicolumn{2}{|c|}{00000111} & \multicolumn{2}{|c|}{00001011}  & \multicolumn{2}{|c|}{00001101} & \multicolumn{2}{|c|}{00010011} &
\multicolumn{2}{|c|}{00010101} & \multicolumn{2}{|c|}{00011001} & \multicolumn{2}{|c|}{00100101}\\
\hline
$\mathcal{I}_i$ & $\mathcal{S}_i$ & $\mathcal{I}_i$ & $\mathcal{S}_i$ & $\mathcal{I}_i$ & $\mathcal{S}_i$ & $\mathcal{I}_i$ & $\mathcal{S}_i$ & 
$\mathcal{I}_i$ & $\mathcal{S}_i$ & $\mathcal{I}_i$ & $\mathcal{S}_i$ & $\mathcal{I}_i$ & $\mathcal{S}_i$\\
\hline
7 & 7 & 11 & 11 & 13 & 13 & 19 & 19 & 21 & 21 & 25 & 25 & 37 & 37\\
\hline
14 & 21 & 22 & 33 & 26 & 39 & 38 & 57 & 42 & 63 & 35 & 60 & 41 & 78\\
\hline
28 & 49 & 44 & 77 & 52 & 91 & 49 & 106 & 69 & 132 & 50 & 110 & 73 & 151\\
\hline
56 & 105 & 88 & 165 & 67 & 158 & 76 & 182 & 81 & 213 & 70 & 180 & 74 & 225\\
\hline
112 & 217 & 97 & 262 & 104 & 262 & 98 & 280 & 84 & 297 & 100 & 280 & 82 & 307\\
\hline
131 & 348 & 133 & 395 & 134 & 396 & 137 & 417 & 138 & 435 & 140 & 420 & 146 & 453\\
\hline
193 & 541 & 176 & 571 & 161 & 557 & 152 & 569 & 162 & 597 & 145 & 565 & 148 & 601\\
\hline
224 & 765 & 194 & 765 & 208 & 765 & 196 & 765 & 168 & 765 & 200 & 765 & 164 & 765\\
\hline
\end{tabular}
\caption{The base-2 orbits and the partial sums in $\mathbb{W}_{3,8}$.}
\label{table:tableSum}
\end{center}
\end{table}

\begin{figure}[ht]
\setlength{\unitlength}{0.85mm}
\begin{picture}(140,100)
\put(85,100){$00100101$}
\put(95,97){\vector(-3,-2){25}}
\put(95,97){\vector(3,-2){25}}

\put(55,75){$00011001$}
\put(65,72){\vector(0,-1){40}}
\put(115,75){$00010101$}
\put(125,72){\vector(0,-1){40}}
\put(125,72){\vector(-3,-2){25}}

\put(85,50){$00010011$}
\put(95,47){\vector(-3,-2){25}}

\put(55,25){$00001101$}
\put(65,22){\vector(3,-2){25}}
\put(115,25){$00001011$}
\put(125,22){\vector(-3,-2){25}}

\put(85,0){$00000111$}
\end{picture}
\caption{The partially ordered set $(\mathbb{W}_{3,8},\prec)$.
If $p$ and $q$ grow large, the poset $(\mathbb{W}_{p,q},\prec)$ grows very complex and it is hard to yield any other general results except the two extremal elements.}
\label{pic1}
\end{figure}
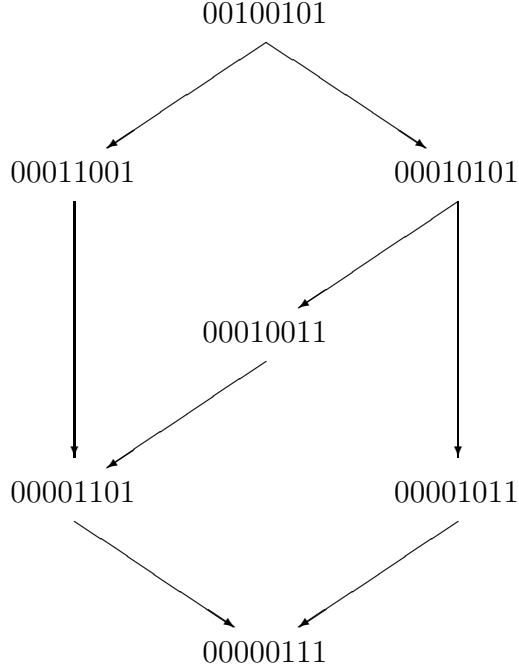

From now on, to make the notation easier, we will suppose that the base-2 expansion $(a_1 a_2 \ldots a_n)_2 = \sum_{i=1}^{n}{a_i 2^{n-i}}$
can contain numbers also different from 0 or 1, i.e. $a_i\in\mathbb{N}$.
If there is a power $a_i^k$ inside a base-2 expansion $(a_1 a_2 \ldots a_n)_2$ then we suppose it means that the number $a_i$ appears $k$ times in a row.
If we have a number with two or more digits then we put parentheses around it.
For example $(0012^3013(14))_2=(0012223013(14))_2=1\cdot2^8+2\cdot2^7+2\cdot2^6+2\cdot2^5+3\cdot2^4+1\cdot2^2+3\cdot2^1+14\cdot 2^0=776.$

We will start with a lemma that states some simple formulas on base-2 numbers, which we will need in the proof of Theorem \ref{t4}.
Notice that for example $(003000)_2=(000600)_2$, $(0040)_2=(0200)_2$, $(010000000)_2=(001111112)_2$ and $(001111111)_2 < (010000000)_2$.
The next lemma uses these kind of facts.

\begin{lemma}\label{l3}
1) $(040^{q-2})_2 = (0160^{q-3})_2 = (01280^{q-4})_2 = \ldots = (0123\ldots (q-3)(q-2)(2q))_2$.

2) $(0123\ldots (q-3)(q-2)(2q))_2 > (0123\ldots (p-2)(p-1) p^{q-2p+1}(p-1)(p-2)\ldots 321)_2,$ if $p\leq q-p$.

3) $(0123\ldots (q-3)(q-2)(2q))_2 > (0123\ldots (q-p-1) (q-p)^{2p-q+1}(q-p-1)\ldots 321)_2,$ if $p > q-p$.

4) $(040^{q-2})_2 = (0320^{q-3})_2 = (03120^{q-4})_2 = (02320^{q-4})_2$.

5) $(02311210^{q-7})_2 = (02312010^{q-7})_2 = (02320010^{q-7})_2 > (02320^{q-4})_2$.

6) $(0220^{q-3})_2 = (0140^{q-3})_2 = (012^{q-3}4)_2 > (012^{q-3}1)_2$.

7) $(02310^{q-4})_2 = (01270^{q-4})_2 = (01234^{q-5}8)_2 > (01234^{q-7}321)_2$.

8) $(022232110^{q-8})_2 = (022240110^{q-8})_2 > (022240^{q-5})_2 = (022400^{q-5})_2 = (02320^{q-4})_2$.

9) $(02222110^{q-7})_2 = (02302110^{q-7})_2 = (02310110^{q-7})_2 > (02310^{q-4})_2 > (01234^{q-7}321)_2$.

10) $(0221210^{q-6})_2 > (022120^{q-5})_2 = (0150^{q-3})_2 = (0123^{q-4}6)_2 > (0123^{q-5}21)_2$.

11) $(021120^{q-5})_2 = (013120^{q-5})_2 = (014000^{q-5})_2 = (012^{q-3}4)_2 > (012^{q-3}1)_2$.
\end{lemma}

\begin{table}[ht]
\begin{center}
\begin{tabular}{|c|l|c|l|}
\hline
$i$ & \multicolumn{1}{|c|}{$u_i$} & $> / <$ & \multicolumn{1}{|c|}{$w_i$} \\
\hline
1 & $0^{q-p}1^{p}$ & $<$ & $0^{q-p-1}1w'_1$ \\
\hline
2 & $0^{q-p-1}1^{p}0$ & $<$ & $0^{q-p-2}1w'_2$ \\
\hline
$\ldots$ & \multicolumn{1}{|c|}{$\ldots$} & \multicolumn{1}{|c|}{$\ldots$} &  \multicolumn{1}{|c|}{$\ldots$}  \\
\hline
$q-p-3$ & $00001^{p}0^{q-p-4}$ & $<$ &  $0001w'_{q-p-3}$   \\
\hline
$q-p-2$ & $0001^{p}0^{q-p-3}$ & $<$ &  $001w'_{q-p-2}$ \\
\hline
$q-p-1$ & $001^{p}0^{q-p-2}$ & $<$ &  $01w'_{q-p-1}$ \\
\hline
$q-p$ & $01^{p}0^{q-p-1}$ & $>$ &  $01w'_{q-p}$ \\
\hline
$q-p+1$ & $10^{q-p}1^{p-1}$ & $<$ &  $10w'_{q-p+1}$ \\
\hline
$q-p+2$ & $110^{q-p}1^{p-2}$ & $>$ &  $10w'_{q-p+2}$ \\
\hline
$q-p+3$ & $1110^{q-p}1^{p-3}$ & $>$ &  $110w'_{q-p+3}$ \\
\hline
$\ldots$ & \multicolumn{1}{|c|}{$\ldots$} & \multicolumn{1}{|c|}{$\ldots$} & \multicolumn{1}{|c|}{$\ldots$} \\
\hline
$q-1$ & $1^{p-1}0^{q-p}1$ & $>$ &  $1^{p-2}0w'_{q-1}$ \\
\hline
$q$ & $1^{p}0^{q-p}$ & $>$ &  $1^{p-1}0w'_{q}$ \\
\hline
\end{tabular}
\caption{Orbits of $u$ and $w$ from Theorem \ref{t4}.}
\label{table:table3}
\end{center}
\end{table}

\begin{theorem}\label{t4}
For any integers $1\leq p < q$, the most unbalanced orbit $u=0^{q-p}1^p\in\mathbb{W}_{p,q}$
is the greatest element in $(\mathbb{W}_{p,q},\prec)$. In other words, for any $w\in\mathbb{W}_{p,q}$,
$$ \mathcal{S}_i(u) \leq \mathcal{S}_i(w)\quad {\rm for\ all}\ 1\leq i\leq q.$$
\end{theorem}
\begin{proof}
We mark $w=0^{r_1}1^{s_1}0^{r_2}1^{s_2}\ldots 0^{r_n}1^{s_n}$, where $\sum_{i=1}^{n}{s_i}=p$, $\sum_{i=1}^{n}{r_i}=q-p$, $n\geq2$ and $\forall i: r_i,s_i > 0$.
The orbits of $w$ and $u$ are marked with $(w_1,\ldots,w_q)$ and $(u_1,\ldots,u_q)$.

We get Table \ref{table:table3} by writing the orbits of $u$ and $w$ in (lexicographic) order.
There are $p$ number of ones and $n\geq2$ so words $w_{q-p+1}$ and $w_{q-p+2}$ start with $10$ (the rest of the word is marked with $w_i'$).
For the same reasons words $w_{q-p}$ and $w_{q-p-1}$ start with $01$.

For words from $w_1$ to $w_{q-p-2}$ we get that $w_i$ cannot be smaller than a word which we get by increasing the number of zeros in front of the word by one, starting from $w_{q-p-1}$.
This is because the number of zeros in front of the word cannot increase with more than one, when moving one word upwards, and we clearly get a larger word if it does not increase.
Similarly, for words from $w_{q-p+3}$ to $w_q$ we get that $w_i$ cannot be larger than a word which we get by increasing the number of ones in front of the word by one, starting from $w_{q-p+2}$.
We will suppose that all these words $w_i$ start as described.

Now we see that $(u_i)_2 < (w_i)_2$ for $1\leq i \leq q-p-1, i=q-p+1$, because we estimated the words from $w_1$ to $w_{q-p-2}$ to be the smallest possible.
Similarly $(u_i)_2 > (w_i)_2$ for $q-p+2\leq i \leq q, i=q-p$, because we estimated the words from $w_{q-p+3}$ to $w_q$ to be the largest possible.

Now we get that $\mathcal{S}_i(u) \leq \mathcal{S}_i(w)$ for $1\leq i\leq q-p-1$.
If we suppose that $\mathcal{S}_{q-p}(u) \leq \mathcal{S}_{q-p}(w)$ then we clearly get that also $\mathcal{S}_{q-p+1}(u) \leq \mathcal{S}_{q-p+1}(w)$, since $(u_{q-p+1})_2 < (w_{q-p+1})_2$.
We already deduced in the preliminaries that $\mathcal{S}_q(u) = \mathcal{S}_q(w) = (2^{q}-1)p$.
Because $(u_i)_2 > (w_i)_2$ for $q-p+2\leq i \leq q$, we get that $\mathcal{S}_i(u) \leq \mathcal{S}_i(w)$ for $q-p+2\leq i\leq q$.

The only thing we need to prove anymore is our assumption $\mathcal{S}_{q-p}(u) \leq \mathcal{S}_{q-p}(w)$ in the previous paragraph.
Direct calculation gives:
$$ \mathcal{S}_{q-p}(u) = \sum^{q-p}_{k=1}{\mathcal{I}_k(u)} = (0123\ldots (p-2)(p-1) p^{q-2p+1}(p-1)(p-2)\ldots 321)_2\ \textrm{if}\ p\leq q-p, $$
$$ \mathcal{S}_{q-p}(u) = \sum^{q-p}_{k=1}{\mathcal{I}_k(u)} = (0123\ldots (q-p-1) (q-p)^{2p-q+1}(q-p-1)\ldots 321)_2\ \textrm{if}\ p > q-p. $$
Now we divide the proof into three cases: 1) $n\geq 4$, 2) $n=3$ and 3) $n=2$.

\begin{table}[ht]
\begin{center}
\begin{tabular}{|l|l|l|l|l|l|l|l|l|}
\hline
\multicolumn{1}{|c|}{$i$} & \multicolumn{8}{|c|}{Prefixes of $w_i$} \\ \cline{2-9}
& \multicolumn{1}{|c|}{1} & \multicolumn{1}{|c|}{2.1} & \multicolumn{1}{|c|}{2.2} & \multicolumn{1}{|c|}{2.3} & \multicolumn{1}{|c|}{2.4} & \multicolumn{1}{|c|}{3.1}
& \multicolumn{1}{|c|}{3.2} & \multicolumn{1}{|c|}{3.3} \\
\hline
$q-p-4$ & & 001 & & & & & &\\
\hline
$q-p-3$ & 01 & 001 & 001 & & & 0011 & & \\
\hline
$q-p-2$ & 01 & 01 & 01 & 0101 & 010101 & 0011 & 0011011 & \\
\hline
$q-p-1$ & 01 & 01 & 0101 & 0101 & 010101 & 011 & 011 & 011 \\
\hline
$q-p$ & 01 & 01 & 0101 & 011 & 010101 & 011 & 011011 & 011 \\
\hline
\end{tabular}
\caption{Prefixes of $w_i$ from cases 1-3.3.}
\label{table:table4}
\end{center}
\end{table}

1) Because $n\geq 4$ the words $w_{q-p},w_{q-p-1},w_{q-p-2}$ and $w_{q-p-3}$ start with $01$.
It is enough to take only these four words to the partial sum $\mathcal{S}_{q-p}(w)$ and even suppose that the remaining parts of these four words are zeros.
We get that $$\mathcal{S}_{q-p}(w) \geq \sum^{q-p}_{k=q-p-3}{\mathcal{I}_k(w)} \geq (040^{q-2})_2 = (0123\ldots (q-2)(2q))_2 > \mathcal{S}_{q-p}(u),$$
where the equality comes from Lemma \ref{l3}, 1 and the last inequality from Lemma \ref{l3}, 2\&3.

2) The case $n=3$ is similar to the previous one. We divide it into four subcases depending on the values of $r_i$ and $s_i$.
Because $n=3$ the words $w_{q-p}, w_{q-p-1}$ and $w_{q-p-2}$ start with 01 in all cases.
From now on, we will use Lemma \ref{l3} without stating it explicitly.

2.1) $\exists i,j$ ($i\neq j$): $r_i,r_j\geq2$.
This means there are at least two blocks of zeros of length at least 2, which means that the words $w_{q-p-3}$ and $w_{q-p-4}$ starts with 001.
Now it is enough to take only five words to the partial sum $\mathcal{S}_{q-p}(w)$ and suppose that the remaining parts are zeros.
We get that $$\mathcal{S}_{q-p}(w) \geq \sum^{q-p}_{k=q-p-4}{\mathcal{I}_k(w)} \geq (0320^{q-3})_2 > \mathcal{S}_{q-p}(u).$$

2.2) $\exists! i: r_i\geq2$.
Because we have two blocks of zeros of length 1, we get that the words $w_{q-p},w_{q-p-1}$ start with $011,011$ or $011,0101$ or $0101,0101$ (depending on the values of $s_i$).
We can estimate the partial sum $\mathcal{S}_{q-p}(w)$ downwards so we suppose they start with $0101,0101$.
Since $r_i\geq2$, the word $w_{q-p-3}$ starts with $001$.
We get that $$\mathcal{S}_{q-p}(w) \geq \sum^{q-p}_{k=q-p-3}{\mathcal{I}_k(w)} \geq (03120^{q-4})_2 > \mathcal{S}_{q-p}(u).$$

2.3) $r_1,r_2,r_3=1$, $\exists i:$ $s_i\geq2$. Notice that $q-p=r_1+r_2+r_3=3$.
Because there is at least one block of ones of length at least 2 and $r_1,r_2,r_3=1$, we get that the words $w_{q-p},w_{q-p-1},w_{q-p-2}$ start with $011,011,011$ or $011,011,0101$ or $011,0101,0101$.
From these, we again choose the smallest ones $011,0101,0101$ in order to estimate $\mathcal{S}_{q-p}(w)$ downwards.
We get that $$\mathcal{S}_{q-p}(w) = \sum^{3}_{1}{\mathcal{I}_k(w)} \geq (03120^{q-4})_2 > \mathcal{S}_{q-p}(u).$$

2.4) $\forall i: r_i,s_i=1$, i.e. $u=000111$ and $w=010101$. Trivially we get that
$$\mathcal{S}_{q-p}(w) = (030303)_2 > (012321)_2 = \mathcal{S}_{q-p}(u).$$

3) The case $n=2$ is similarly divided into several subcases depending on the values of $r_i$ and $s_i$.
Because $n=2$ the words $w_{q-p}$ and $w_{q-p-1}$ start with 01 in all cases.

3.1) $s_1,s_2\geq2$ and $r_1,r_2\geq2$.
Because $s_1,s_2\geq2$ the words $w_{q-p}$ and $w_{q-p-1}$ start with $011$, and because $r_1,r_2\geq2$ the words $w_{q-p-2}$ and $w_{q-p-3}$ start with $0011$.
We get that $$\mathcal{S}_{q-p}(w) \geq \sum^{q-p}_{k=q-p-3}{\mathcal{I}_k(w)} \geq (02420^{q-4})_2 > (02320^{q-4})_2 > \mathcal{S}_{q-p}(u).$$

3.2) $s_1,s_2\geq2$ and $r_1=1,r_2\geq2$.
Because $r_1=1$ and $s_1,s_2\geq2$ the word $w_{q-p}$ starts with $0111$ or $011011$ and because $r_2\geq2$ the word $w_{q-p-2}$ starts with $00111$ or $0011011$.
From these we choose the smaller ones $011011$ and $0011011$. The word $w_{q-p-1}$ starts with $011$, since $s_1,s_2\geq2$.
We get that $$\mathcal{S}_{q-p}(w) \geq \sum^{q-p}_{k=q-p-2}{\mathcal{I}_k(w)} \geq (02311210^{q-7})_2 > (02320^{q-4})_2 > \mathcal{S}_{q-p}(u).$$

3.3) $s_1,s_2\geq2$ and $r_1,r_2=1$. Notice that $q-p=r_1+r_2=2$.
Because $s_1,s_2\geq2$ the words $w_1$ and $w_2$ start with $011$.
We get that $$\mathcal{S}_{q-p}(w) = \sum^{2}_{k=1}{\mathcal{I}_k(w)} \geq (0220^{q-3})_2 > (012^{q-3}1)_2 = \mathcal{S}_{q-p}(u).$$

\begin{table}[ht]
\begin{center}
\begin{tabular}{|l|l|l|l|l|l|l|l|l|l|}
\hline
\multicolumn{1}{|c|}{$i$} & \multicolumn{9}{|c|}{Prefixes of $w_i$} \\ \cline{2-10}
& \multicolumn{1}{|c|}{3.4.1-2} & \multicolumn{1}{|c|}{3.5.1-2} & \multicolumn{1}{|c|}{3.5.3} & \multicolumn{1}{|c|}{3.6.1-3} & \multicolumn{1}{|c|}{3.7} & \multicolumn{1}{|c|}{3.8} & \multicolumn{1}{|c|}{3.9} & \multicolumn{1}{|c|}{3.10} & \multicolumn{1}{|c|}{3.11} \\
\hline
$q-p-4$ & 0001 & 00001101 & & 00001011 & & & & & \\
\hline
$q-p-3$ & 001 & 0001101 & & 0001011 & & 001 & 0001 & & \\
\hline
$q-p-2$ & 0011 & 001101 & 001101 & 001011 & & 001 & 001 & 00101 & \\
\hline
$q-p-1$ & 01 & 01 & 01001 & 01011 & 01011 & 01 & 01 & 01001 & 0101 \\
\hline
$q-p$ & 011 & 01101 & 01101 & 011 & 01101 & 01 & 0101 & 01010 & 0101 \\
\hline
\end{tabular}
\caption{Prefixes of $w_i$ from cases 3.4-3.11.}
\label{table:table5}
\end{center}
\end{table}

3.4.1) $s_1=1,s_2\geq2$ and $r_1\geq3,r_2\geq2$ or $r_1\geq2,r_2\geq3$.
We can easily see that the words $w_{q-p},w_{q-p-1},w_{q-p-2},w_{q-p-3},w_{q-p-4}$ start with $01,01,001,001,0001$.
Because $s_2\geq2$ we additionally get that from these words $w_{q-p}$ and $w_{q-p-2}$ start with $011$ and $0011$.
Together we have $$\mathcal{S}_{q-p}(w) \geq \sum^{q-p}_{k=q-p-4}{\mathcal{I}_k(w)} \geq (02320^{q-4})_2 > \mathcal{S}_{q-p}(u).$$

3.4.2) $s_1=1,s_2\geq2$ and $r_1,r_2=2$. Notice that $q-p=r_1+r_2=4$.
This is similar to the previous case except we do not have the word $w_{q-p-4}$.
We get $$\mathcal{S}_{q-p}(w) = \sum^{4}_{k=1}{\mathcal{I}_k(w)} \geq (02310^{q-4})_2 > (01234^{q-7}321)_2 = \mathcal{S}_{q-p}(u).$$

3.5.1) $s_1=1,s_2\geq2$ and $r_1=1,r_2\geq4$.
Because $r_1=1$ and $s_2\geq2$ the word $w_{q-p}$ starts with $0111$ or $01101$. We choose the smaller one $01101$.
Since $r_2\geq4$, the same applies to the words $w_{q-p-2},w_{q-p-3}$ and $w_{q-p-4}$, which are estimated to start with $001101,0001101$ and $00001101$.
We get $$\mathcal{S}_{q-p}(w) \geq \sum^{q-p}_{k=q-p-4}{\mathcal{I}_k(w)} \geq (022232110^{q-8})_2 > (02320^{q-4})_2 > \mathcal{S}_{q-p}(u).$$

3.5.2) $s_1=1,s_2\geq2$ and $r_1=1,r_2=3$. Notice that $q-p=r_1+r_2=4$.
This is similar to the previous case except we do not have the word $w_{q-p-4}$.
We get $$\mathcal{S}_{q-p}(w) = \sum^{4}_{k=1}{\mathcal{I}_k(w)} \geq (02222110^{q-7})_2 > (01234^{q-7}321)_2 = \mathcal{S}_{q-p}(u).$$

3.5.3) $s_1=1,s_2\geq2$ and $r_1=1,r_2=2$. Notice that $q-p=r_1+r_2=3$.
This is similar to the previous case except we do not have the word $w_{q-p-3}$ and we know that the word $w_{q-p-1}$ starts with $01001$.
We get $$\mathcal{S}_{q-p}(w) = \sum^{3}_{k=1}{\mathcal{I}_k(w)} \geq (0221210^{q-6})_2 > (0123^{q-5}21)_2 = \mathcal{S}_{q-p}(u).$$

3.6.1) $s_1=1,s_2\geq2$ and $r_1\geq4,r_2=1$.
Because $s_2\geq2$ the word $w_{q-p}$ starts with $011$. Because $s_1,r_2=1$ and $r_1\geq4$ we know that the words $w_{q-p-1},w_{q-p-2},w_{q-p-3}$ and $w_{q-p-4}$ start with $01011,001011,0001011$ and $00001011$.
We get $$\mathcal{S}_{q-p}(w) \geq \sum^{q-p}_{k=q-p-4}{\mathcal{I}_k(w)} \geq (022232210^{q-8})_2 > (02320^{q-4})_2 > \mathcal{S}_{q-p}(u).$$

3.6.2) $s_1=1,s_2\geq2$ and $r_1=3,r_2=1$. Notice that $q-p=r_1+r_2=4$.
This is similar to the previous case except we do not have the word $w_{q-p-4}$.
We get $$\mathcal{S}_{q-p}(w) = \sum^{4}_{k=1}{\mathcal{I}_k(w)} \geq (02222210^{q-7})_2 > (01234^{q-7}321)_2 = \mathcal{S}_{q-p}(u).$$

3.6.3) $s_1=1,s_2\geq2$ and $r_1=2,r_2=1$. Notice that $q-p=r_1+r_2=3$.
This is similar to the previous case except we do not have the word $w_{q-p-3}$.
We get $$\mathcal{S}_{q-p}(w) = \sum^{3}_{k=1}{\mathcal{I}_k(w)} \geq (0221210^{q-6})_2 > (0123^{q-5}21)_2 = \mathcal{S}_{q-p}(u).$$

3.7) $s_1=1,s_2\geq2$ and $r_1,r_2=1$. Notice that $q-p=r_1+r_2=2$.
Because $s_2\geq2$ and $r_1=1$ the word $w_2$ starts with $0111$ or $01101$ from which we choose the smaller one $01101$. The word $w_1$ starts with $01011$.
We get $$\mathcal{S}_{q-p}(w) = \sum^{2}_{k=1}{\mathcal{I}_k(w)} \geq (021120^{q-5})_2 > (012^{q-3}1)_2 = \mathcal{S}_{q-p}(u).$$

3.8) $s_1,s_2=1$ and $r_1,r_2\geq2$. Notice that $p=s_1+s_2=2$.
Because $r_1,r_2\geq2$ the words $w_{q-p-2}$ and $w_{q-p-3}$ start with $001$.
We get $$\mathcal{S}_{q-p}(w) \geq \sum^{q-p}_{k=q-p-3}{\mathcal{I}_k(w)} \geq (0220^{q-3})_2 > (012^{q-3}1)_2 = \mathcal{S}_{q-p}(u).$$

3.9) $s_1,s_2=1$ and $r_1=1,r_2\geq3$.
Because $s_1,r_2=1$ the word $w_{q-p}$ starts with $0101$ and because $r_2\geq3$ the words $w_{q-p-1},w_{q-p-2}$ and $w_{q-p-3}$ start with $01,001$ and $0001$.
We get $$\mathcal{S}_{q-p}(w) \geq \sum^{q-p}_{k=q-p-3}{\mathcal{I}_k(w)} \geq (02120^{q-4})_2 = (0220^{q-3})_2 > (012^{q-3}1)_2 = \mathcal{S}_{q-p}(u).$$

3.10) $s_1,s_2=1$ and $r_1=1,r_2=2$. Now $u=00011$ and $w=01001$.
We get $$\mathcal{S}_{q-p}(w) = (02112)_2 > (01221)_2 = \mathcal{S}_{q-p}(u).$$

3.11) $s_1,s_2=1$ and $r_1,r_2=1$. Now $u=0011$ and $w=0101$.
We get $$\mathcal{S}_{q-p}(w) = (0202)_2 > (0121)_2 = \mathcal{S}_{q-p}(u).$$
\end{proof}


\section{Product}

In this section we prove the counterpart for Theorem \ref{t2}. We will not need the condition that $q$ and $p$ are coprime.
The following two lemmas state some simple inequalities on base-2 numbers that we need in the proof of Theorem \ref{t5}.
We will suppose that the base-2 expansion $(a_1 a_2 \ldots a_n)_2 = \sum_{i=1}^{n}{a_i 2^{n-i}}$ can contain also rational numbers, i.e. $a_i\in\mathbb{Q}$.
For example $(\frac{1}{2}00\frac{3}{2}0)_2 = (010\frac{3}{2}0)_2 = (01011)_2$ and
$\frac{2}{3}\cdot(0110)_2 = (0\frac{2}{3}\frac{2}{3}0)_2 = (0\frac{6}{3}00)_2 = (1000)_2$.

\begin{lemma}\label{l4}
The following inequalities hold for any $w\in\{0,1\}^*$ and $a,b\geq0$ $(a+b\geq1)$
such that the words on both sides are equally long and have equally many zeros and ones.

1) $2\cdot (0^{b_1} 1^a 0^{b_2})_2 < (0^{b_1-2}1w)_2$, where $b_1\geq 3$ and $b_2\geq 0$.

2) $4\cdot(000001^a 0^b)_2 < (001w)_2$.

3) $\frac{21}{8}\cdot(00001^a 0^b)_2 < (0010101w)_2$.

4) $(0^{b_1} 1^a 0^{b_2})_2 < (0^{b_1-1} 1w)_2$, where $b_1\geq 2$ and $b_2\geq 0$.

5) $\frac{1}{2}\cdot(01^a 0^b)_2 < (01w)_2$.

6) $(10^b 1^a)_2 \leq (1w)_2$.

7) $\frac{2}{3}\cdot(110^b 1^a)_2 < (1w)_2$.

8) $\frac{4}{7}\cdot(1110^b 1^a)_2 < (1w)_2$.

9) $\frac{1}{2}\cdot(1^{a_1}0^b 1^{a_2})_2 < (1w)_2$, where $a_1\geq 1$ and $a_2\geq 0$.
\end{lemma}
\begin{proof}
1) $2\cdot (0^{b_1} 1^a 0^{b_2})_2 = (0^{b_1-1} 1^a 0^{b_2+1})_2 < (0^{b_1-2}10^{a+b_2+1})_2 < (0^{b_1-2}1w)_2$.

2) $4\cdot(000001^a 0^b)_2 = 2\cdot(00001^a 0^{b+1})_2 = (0001^a 0^{b+2})_2 < (0010^{a+b+2})_2 < (001w)_2$.

3) $\frac{21}{8}\cdot(00001^a 0^b)_2 < \frac{21}{8}\cdot(0001 0^{a+b})_2  = (000\frac{21}{8} 0^{a+b})_2 = (0010101 0^{a+b-3})_2 < (0010101w)_2$

4) $(0^{b_1} 1^a 0^{b_2})_2 < (0^{b_1-1} 1 0^{a+b_2})_2 < (0^{b_1-1} 1w)_2$.

5) $\frac{1}{2}\cdot(01^a 0^b)_2 < \frac{1}{2}\cdot(1 0^{a+b})_2 = (01 0^{a+b-1})_2 < (01w)_2$.

6) Trivial.

7) $\frac{2}{3}\cdot(110^b 1^a)_2 = (\frac{2}{3}\frac{2}{3}0^b(\frac{2}{3})^a)_2 = (100^b(\frac{2}{3})^a)_2 < (1w)_2$.

8) $\frac{4}{7}\cdot(1110^b 1^a)_2 = (\frac{4}{7}\frac{4}{7}\frac{4}{7}0^b(\frac{4}{7})^a)_2 =
(\frac{4}{7}\frac{6}{7}00^b(\frac{4}{7})^a)_2 = (1000^b(\frac{4}{7})^a)_2 < (1w)_2$.

9) $\frac{1}{2}\cdot(1^{a_1}0^b 1^{a_2})_2 < \frac{1}{2}\cdot(1^{a_1+b+a_2})_2 < \frac{1}{2}\cdot(20^{a_1+b+a_2-1})_2 = (10^{a_1+b+a_2-1})_2 < (1w)_2$.
\end{proof}

\begin{lemma}\label{l5}
The following inequalities hold for any $w\in\{0,1\}^*$ and $a,b\geq0$ $(a+b\geq1)$
such that the words on both sides are equally long and have equally many zeros and ones.

1) $\frac{3}{2}\cdot(0^{b_1+1} 1^a 0^{b_2})_2 < (0^{b_1}11w)_2$, where $b_1\geq 1$ and $b_2\geq 0$.

2) $\frac{11}{8}\cdot(0^{b_1+1} 1^a 0^{b_2})_2 < (0^{b_1}1011w)_2$, where $b_1\geq 1$ and $b_2\geq 0$.

3) $\frac{8}{3}\cdot(000011 0^b)_2 < (001w)_2$.

4) $\frac{13}{8}\cdot(0001^a 0^b)_2 < (001101w)_2$.

5) $\frac{5}{3}\cdot(00011 0^b)_2 < (00101w)_2$.

6) $\frac{3}{4}\cdot(01^a 0^b)_2 < (011w)_2$.

7) $\frac{2}{3}\cdot(0^{b_1} 11 0^{b_2} 1^{a})_2 < (0^{b_1}1w)_2$, where $a,b_1,b_2\geq 0$.

8) $\frac{3}{4}\cdot(1^{a_1} 0^{b} 1^{a_2})_2 < (11w)_2$, where $a_1,b\geq 1$ and $a_2\geq 0$.

9) $\frac{5}{6}\cdot(0^{b_1} 11 0^{b_2} 1^{a})_2 < (0^{b_1}101w)_2$, where $a,b_1\geq 0$ and $b_2\geq 1$.
\end{lemma}
\begin{proof}
1) $\frac{3}{2}\cdot(0^{b_1+1} 1^a 0^{b_2})_2 < \frac{3}{2}\cdot(0^{b_1}10^{a+b_2})_2 = (0^{b_1}\frac{3}{2}0^{a+b_2})_2 = (0^{b_1}110^{a+b_2-1})_2 < (0^{b_1}11w)_2$.

2) $\frac{11}{8}\cdot(0^{b_1+1} 1^a 0^{b_2})_2 < \frac{11}{8}\cdot(0^{b_1}10^{a+b_2})_2 = (0^{b_1}\frac{11}{8}0^{a+b_2})_2 = (0^{b_1}10110^{a+b_2-3})_2 < (0^{b_1}1011w)_2$.

3) $\frac{8}{3}\cdot(000011 0^b)_2 = (0000\frac{8}{3}\frac{8}{3} 0^b)_2 = (0000\frac{12}{3}0 0^b)_2 = (001000 0^b)_2 < (001w)_2$.

4) $\frac{13}{8}\cdot(0001^a 0^b)_2 < \frac{13}{8}\cdot(0010^{a+b})_2 = (00\frac{13}{8}0^{a+b})_2 = (0011010^{a+b-3})_2 < (001101w)_2$.

5) $\frac{5}{3}\cdot(00011 0^b)_2 = (000\frac{5}{3}\frac{5}{3}0^b)_2 = (000\frac{6}{3}10^b)_2 = (001010^b)_2 < (00101w)_2$.

6) $\frac{3}{4}\cdot(01^a 0^b)_2 < \frac{3}{4}\cdot(1 0^{a+b})_2 = (\frac{3}{4} 0^{a+b})_2 = (0\frac{6}{4} 0^{a+b-1})_2 = (011 0^{a+b-2})_2 < (011w)_2$.

7) $\frac{2}{3}\cdot(0^{b_1} 11 0^{b_2} 1^{a})_2 = (0^{b_1} \frac{2}{3}\frac{2}{3} 0^{b_2} \frac{2}{3}^{a})_2 = (0^{b_1} 10 0^{b_2} \frac{2}{3}^{a})_2 < (0^{b_1}1w)_2$

8) $\frac{3}{4}\cdot(1^{a_1} 0^{b} 1^{a_2})_2 < \frac{3}{4}\cdot(1 0^{a_1+b} 1^{a_2})_2 = (\frac{3}{4} 0^{a_1+b} \frac{3}{4}^{a_2})_2 = (011 0^{a_1+b-2} \frac{3}{4}^{a_2})_2$ $< (11w)_2$
(notice that the length of the base-2 expansion changes after the first and last inequality).

9) $\frac{5}{6}\cdot(0^{b_1} 11 0^{b_2} 1^{a})_2 = (0^{b_1} \frac{5}{6}\frac{5}{6} 0^{b_2} \frac{5}{6}^{a})_2 = (0^{b_1} 1\frac{3}{6} 0^{b_2} \frac{5}{6}^{a})_2 = (0^{b_1} 101 0^{b_2-1} \frac{5}{6}^{a})_2 < (0^{b_1}101w)_2$.
\end{proof}

\vspace{0.3cm}
\begin{table}[ht]
\begin{center}
\begin{tabular}{|c|l|l|l|c|c|}
\hline
$i$ & \multicolumn{1}{|c|}{$u_i$} & \multicolumn{2}{|c|}{$w_i$} & \multicolumn{2}{|c|}{Multiplier($i$)}\\ \cline{3-6}
& & \multicolumn{1}{|c|}{Case 1.1} & \multicolumn{1}{|c|}{Case 1.2} & \multicolumn{1}{|c|}{Case 1.1} & \multicolumn{1}{|c|}{Case 1.2} \\
\hline
1 & $0^{q-p}1^{p}$ & \multicolumn{2}{|l|}{$0^{n}1w'_1$\quad ($n\leq q-p-2$)} & \multicolumn{2}{|c|}{2} \\
\hline
2 & $0^{q-p-1}1^{p}0$ & \multicolumn{2}{|l|}{$0^{n}1w'_2$\quad ($n\leq q-p-3$)} & \multicolumn{2}{|c|}{2} \\
\hline
$\ldots$ & \multicolumn{1}{|c|}{$\ldots$} & \multicolumn{2}{|c|}{$\ldots$} & \multicolumn{2}{|c|}{$\ldots$} \\
\hline
$q-p-5$ & $0000001^{p}0^{q-p-6}$ & \multicolumn{2}{|l|}{$0^{n}1w'_{q-p-5}$\quad ($n\leq 4$)} & \multicolumn{2}{|c|}{2} \\
\hline
$q-p-4$ & $000001^{p}0^{q-p-5}$ & $001w'_{q-p-4}$ & $0001w'_{q-p-4}$ & \multicolumn{1}{@{\ \ \ \ \ } c @{\ \ \ \ \ }|}{4} & 2 \\
\hline
$q-p-3$ & $00001^{p}0^{q-p-4}$ & $001w'_{q-p-3}$ & $0010101w'_{q-p-3}$ & 2 & $\frac{21}{8}$ \\
\hline
$q-p-2$ & $0001^{p}0^{q-p-3}$ & \multicolumn{2}{|l|}{$01w'_{q-p-2}$} & \multicolumn{2}{|c|}{2} \\
\hline
$q-p-1$ & $001^{p}0^{q-p-2}$ & \multicolumn{2}{|l|}{$01w'_{q-p-1}$} & \multicolumn{2}{|c|}{1} \\
\hline
$q-p$ & $01^{p}0^{q-p-1}$ & \multicolumn{2}{|l|}{$01w'_{q-p}$} & \multicolumn{2}{|c|}{$\frac{1}{2}$} \\
\hline
$q-p+1$ & $10^{q-p}1^{p-1}$ & \multicolumn{2}{|l|}{$1w'_{q-p+1}$} & \multicolumn{2}{|c|}{1} \\
\hline
$q-p+2$ & $110^{q-p}1^{p-2}$ & \multicolumn{2}{|l|}{$1w'_{q-p+2}$} & \multicolumn{2}{|c|}{$\frac{2}{3}$} \\
\hline
$q-p+3$ & $1110^{q-p}1^{p-3}$ & \multicolumn{2}{|l|}{$1w'_{q-p+3}$} & \multicolumn{2}{|c|}{$\frac{4}{7}$} \\
\hline
$q-p+4$ & $11110^{q-p}1^{p-4}$ & \multicolumn{2}{|l|}{$1w'_{q-p+4}$} & \multicolumn{2}{|c|}{$\frac{1}{2}$} \\
\hline
$\ldots$ & \multicolumn{1}{|c|}{$\ldots$} & \multicolumn{2}{|c|}{$\ldots$} & \multicolumn{2}{|c|}{$\ldots$} \\
\hline
$q-1$ & $1^{p-1}0^{q-p}1$ & \multicolumn{2}{|l|}{$1w'_{q-1}$} & \multicolumn{2}{|c|}{$\frac{1}{2}$} \\
\hline
$q$ & $1^{p}0^{q-p}$ & \multicolumn{2}{|l|}{$1w'_{q}$} & \multicolumn{2}{|c|}{$\frac{1}{2}$} \\
\hline
\end{tabular}
\caption{Case 1 in Theorem \ref{t5}.}
\label{table:table1}
\end{center}
\end{table}
\vspace{0.3cm}

The idea of the proof of Theorem \ref{t5} is to multiply the base-2 expansions of the words in the orbit $(u_1,\ldots,u_q)$ of the most unbalanced word $u=0^{q-p}1^p$
with some number so that the base-2 expansion of the corresponding word in the orbit $(w_1,\ldots,w_q)$ of any other word $w\in \mathbb{W}_{p,q}$ is larger.
If the product of all the multipliers is at least one then we get that the product of $u$ is smaller than the product of $w$.
Table \ref{table:table1} tells what are the multipliers for each word in case 1 of the proof of Theorem \ref{t5}.
If we multiply the base-2 expansion of $u_i$ with Multiplier($i$) we get smaller number than the base-2 expansion of $w_i$. 
We get that the product of the multipliers really is at least one:
in case 1.1 $\prod_{i=1}^{q}{\textrm{Multiplier}(i)}=2^{q-p-5}\cdot 4\cdot 2\cdot 2\cdot 1 \cdot \frac{1}{2}\cdot 1\cdot \frac{2}{3}\cdot \frac{4}{7}\cdot \frac{1}{2}^{p-3}=\frac{32}{21}\cdot 2^{(q-p)-p-1}$
and in case 1.2 $\prod_{i=1}^{q}{\textrm{Multiplier}(i)}=2^{q-p-4}\cdot \frac{21}{8}\cdot 2\cdot 1 \cdot \frac{1}{2}\cdot 1\cdot \frac{2}{3}\cdot \frac{4}{7}\cdot \frac{1}{2}^{p-3}=2^{(q-p)-p-1}$,
where $(q-p)-p-1\geq 0$ because we will suppose that there are more zeros than ones, i.e. $p<q-p$.
The theorem is probably true even without the assumption that there are more zeros than ones, but it would be more difficult to prove.

\begin{theorem}\label{t5}
Suppose $1\leq p < q$ are integers such that $p<q-p$. For $w\in \mathbb{W}_{p,q}$ the product
$P(w)=\prod^{q}_{i=1}{\mathcal{I}_i(w)}$ is minimized precisely when $w=0^{q-p}1^p$.
\end{theorem}
\begin{proof}
We mark $u=0^{q-p}1^p$ and $w=0^{r_1}1^{s_1}0^{r_2}1^{s_2}\ldots 0^{r_n}1^{s_n}$, where $\sum_{i=1}^{n}{s_i}=p$, $\sum_{i=1}^{n}{r_i}=q-p$, $n\geq2$ and $\forall i: r_i,s_i > 0$.
The orbits of $w$ and $u$ are marked with $(w_1,\ldots,w_q)$ and $(u_1,\ldots,u_q)$.
Our goal is to prove that $P(u)<P(w)$. We divide the proof into two cases: 1) $n\geq3$ and 2) $n=2$.

1) We divide this case into two subcases: 1.1) $\exists i,j$ ($i\neq j$): $r_i,r_j\geq2$ and 1.2) $\exists! i: r_i\geq2$.
Notice that at least one $r_i$ has to be at least two because otherwise there would not be more zeros than ones.

We get Table \ref{table:table1} by writing the orbits of $u$ and $w$ in (lexicographic) order.
There are $p$ number of ones so the words from $w_q$ to $w_{q-p}$ start with the letter $1$ (the rest of the word is marked with $w_i'$).
Because $n\geq3$ the next three words from $w_{q-p-1}$ to $w_{q-p-3}$ start with $01$.
In case 1.1 there are at least two blocks of zeros of length at least two, which means that the next two words $w_{q-p-3}$ and $w_{q-p-4}$ can start with $01,01$ or $01,001$ or $001,001$.
We suppose that the words start with $001,001$ because that makes the product $P(w)$ smallest.
In case 1.2 there is only one block of zeros which is of length at least 2, which means that the word $w_{q-p-3}$ starts with $01,0011,001011$ or $0010101$.
Similar to the case 1.1, we suppose that it starts with $0010101$ because that makes the product $P(w)$ smallest.

The number of zeros in front of the word cannot increase with more than one, when moving one word upwards.
Since $(10^a)_2>(01^a)_2$, we get the smallest possible $w_i$ for the rest by doing exactly that.

From Lemma \ref{l4} we now get directly the following inequalities:
\begin{center}
\begin{tabular}{r @{\ } l}
$2\cdot(u_{q-p-i})_2$&$< (w_{q-p-i})_2\ \textrm{for every}\ 2\leq i\leq q-p-1$ \\
$4\cdot(u_{q-p-4})_2$&$< (w_{q-p-4})_2$ (case 1.1) \\
$21/8\cdot(u_{q-p-3})_2$&$< (w_{q-p-3})_2$ (case 1.2) \\
$(u_{q-p-1})_2$&$< (w_{q-p-1})_2$ \\
$1/2\cdot(u_{q-p})_2$&$< (w_{q-p})_2$ \\
$(u_{q-p+1})_2$&$< (w_{q-p+1})_2$ \\
$2/3\cdot(u_{q-p+2})_2$&$< (w_{q-p+2})_2$ \\
$4/7\cdot(u_{q-p+3})_2$&$< (w_{q-p+3})_2$ \\
$1/2\cdot(u_{q-p+i})_2$&$< (w_{q-p+i})_2\ \textrm{for every}\ 4\leq i\leq p.$ \\
\end{tabular}
\end{center}

We already calculated that the products of the multipliers are at least one:
$2^{q-p-5}\cdot 4\cdot 2\cdot 2\cdot 1 \cdot \frac{1}{2}\cdot 1\cdot \frac{2}{3}\cdot \frac{4}{7}\cdot \frac{1}{2}^{p-3}=\frac{32}{21}\cdot 2^{(q-p)-p-1}>1$
and $2^{q-p-4}\cdot \frac{21}{8}\cdot 2\cdot 1 \cdot \frac{1}{2}\cdot 1\cdot \frac{2}{3}\cdot \frac{4}{7}\cdot \frac{1}{2}^{p-3}=2^{(q-p)-p-1}\geq1$,
where $(q-p)-p-1\geq 0$ because $p<q-p$. From these facts we get our claim:

1.1) $ P(u)=\prod_{i=1}^{q}{(u_{i})_2} < \prod_{i=1}^{q-p-5}[2(u_{i})_2]\cdot 4(u_{q-p-3})_2\cdot 2(u_{q-p-2})_2(u_{q-p-1})_2 \cdot 1/2(u_{q-p})_2 $
$ (u_{q-p+1})_2 \cdot 2/3(u_{q-p+2})_2\cdot 4/7(u_{q-p+3})_2\cdot \prod_{i=q-p+4}^{q}{1/2(u_{i})_2} < \prod_{i=1}^{q}{(w_{i})_2} = P(w). $

1.2) $ P(u)=\prod_{i=1}^{q}{(u_{i})_2} \leq \prod_{i=1}^{q-p-4}[2(u_{i})_2]\cdot 21/8(u_{q-p-3})_2\cdot 2(u_{q-p-2})_2(u_{q-p-1})_2 \cdot 1/2(u_{q-p})_2 $
$ (u_{q-p+1})_2 \cdot 2/3(u_{q-p+2})_2\cdot 4/7(u_{q-p+3})_2\cdot \prod_{i=q-p+4}^{q}{1/2(u_{i})_2} < \prod_{i=1}^{q}{(w_{i})_2} = P(w). $

\vspace{0.3cm}
\begin{table}[ht]
\begin{center}
\begin{tabular}{|p{1.5cm}|l|l|l|l|l|l|c|c|c|c|c|}
\hline
\multicolumn{1}{|c|}{$i$} & \multicolumn{1}{|c|}{$u_i$} & \multicolumn{5}{|c|}{Prefixes of $w_i$} & \multicolumn{5}{|c|}{Multiplier($i$)}\\ \cline{3-12}
& & \multicolumn{1}{|c|}{2.1} & \multicolumn{1}{|c|}{2.2} & \multicolumn{1}{|c|}{2.3} & \multicolumn{1}{|c|}{2.4} & \multicolumn{1}{|c|}{2.5} & 2.1 & 2.2 & 2.3 & 2.4 & 2.5 \\
\hline
\multicolumn{1}{|c|}{1} & $0^{q-p}1^{p}$ & $0^{n}1$ & $0^{n}1$ & $0^{n}1101$ & $0^{n}1011$ & $0^{n}101$ & 2 & $\frac{8}{3}$ & $\frac{13}{8}$ & $\frac{11}{8}$ & $\frac{5}{3}$ \\
\hline
\multicolumn{1}{|c|}{$\ldots$} & \multicolumn{1}{|c|}{$\ldots$} & \multicolumn{1}{|c|}{$\ldots$} & \multicolumn{1}{|c|}{$\ldots$} & \multicolumn{1}{|c|}{$\ldots$} &
\multicolumn{1}{|c|}{$\ldots$} & \multicolumn{1}{|c|}{$\ldots$} & $\ldots$ & $\ldots$ & $\ldots$ & $\ldots$ & $\ldots$ \\
\hline
$q-p-3$ & $00001^{p}0^{q-p-4}$ & 001 & 001 & \multicolumn{1}{|c|}{$\ldots$} & \multicolumn{1}{|c|}{$\ldots$} & \multicolumn{1}{|c|}{$\ldots$} & 2 & $\frac{8}{3}$ & $\ldots$ & $\ldots$ & $\ldots$ \\
\hline
$q-p-2$ & $0001^{p}0^{q-p-3}$ & 0011 & 001 & 001101 & \multicolumn{1}{|c|}{$\ldots$} & 00101 & $\frac{3}{2}$ & 1 & $\frac{13}{8}$ & $\ldots$ & $\frac{5}{3}$ \\
\hline
$q-p-1$ & $001^{p}0^{q-p-2}$ & 01 & 01 & 01 & 01011 & 01 & 1 & 1 & 1 & $\frac{11}{8}$ & 1 \\
\hline
$q-p$ & $01^{p}0^{q-p-1}$ & 011 & 01 & 011 & 011 & 0101 & $\frac{3}{4}$ & $\frac{2}{3}$ & $\frac{3}{4}$ & $\frac{3}{4}$ & $\frac{5}{6}$ \\
\hline
$q-p+1$ & $10^{q-p}1^{p-1}$ & 10 & 10 & 10 & 10 & 10 & 1 & 1 & 1 & 1 & 1 \\
\hline
$q-p+2$ & $110^{q-p}1^{p-2}$ & 10 & 10 & 101 & 101 & 101 & $\frac{2}{3}$ & $\frac{2}{3}$ & $\frac{5}{6}$ & $\frac{5}{6}$ & $\frac{5}{6}$ \\
\hline
$q-p+3$ & $1110^{q-p}1^{p-3}$ & 11 &  & 11 & 11 &  & $\frac{3}{4}$ &  & $\frac{3}{4}$ & $\frac{3}{4}$ &  \\
\hline
\multicolumn{1}{|c|}{$\ldots$} & \multicolumn{1}{|c|}{$\ldots$} & \multicolumn{1}{|c|}{$\ldots$} &  & \multicolumn{1}{|c|}{$\ldots$} & \multicolumn{1}{|c|}{$\ldots$} &  & $\ldots$ &  & $\ldots$ & $\ldots$ &  \\
\hline
\multicolumn{1}{|c|}{$q$} & $1^{p}0^{q-p}$ & 11 &  & 11 & 11 &  & $\frac{3}{4}$ &  & $\frac{3}{4}$ & $\frac{3}{4}$ &  \\
\hline
\end{tabular}
\caption{Case 2 in Theorem \ref{t5}.}
\label{table:table2}
\end{center}
\end{table}
\vspace{0.3cm}

2) This case is similar to the previous one. We divide it into five subcases depending on the values of $r_1,r_2,s_1$ and $s_2$.
Notice that case $r_1,r_2=1$ is impossible because then we would have $2\leq s_1+s_2=p<q-p=r_1+r_2=2$.

2.1) $r_1,r_2\geq 2$ and $(s_1,s_2)\neq (1,1)$

2.2) $r_1,r_2\geq 2$ and $s_1,s_2=1$

2.3) $r_1=1,r_2\geq 2$ and $s_1\geq 1,s_2\geq 2$

2.4) $r_1=1,r_2\geq 2$ and $s_1\geq 2,s_2\geq 1$

2.5) $r_1=1,r_2\geq 2$ and $s_1,s_2=1$.

We get Table \ref{table:table2} by using the same kind of reasoning as in case 1 (the suffixes $w_i'$ of $w_i$ have been left out to save space):

There are $p$ number of ones and $n=2$ so the words from $w_q$ to $w_{q-p+3}$ start with $11$ and the words $w_{q-p+2}$ and $w_{q-p+1}$ start with $10$. 
In cases 2.3, 2.4 and 2.5 we have $r_1=1$ so we additionally know that the word $w_{q-p+2}$ starts with $101$.

There are $q-p$ number of zeros and $n=2$ so the words $w_{q-p}$ and $w_{q-p-1}$ start with $01$.
In addition, in cases 2.1, 2.3 and 2.4 we have $s_1$ or $s_2\geq2$, which means the word $w_{q-p}$ starts with $011$.
In case 2.5 we have $s_1,s_2 = 1$, which means the word $w_{q-p}$ starts with $0101$.
In addition, in case 2.4 we have $r_1=1$ and $s_1\geq 2$, which means the word $w_{q-p-1}$ starts with either $011$ or $01011$, from which we choose the smaller one $01011$.

In cases 2.1 and 2.3 the word $w_{q-p-2}$ starts with $0011$ because $r_1,s_1\geq 2$ or $r_2,s_2\geq 2$.
In addition, in case 2.3 we have $r_1=1$, which means it starts with $00111$ or $001101$, from which we choose the smaller one $001101$.
In cases 2.2 and 2.5 the word $w_{q-p-2}$ starts with $001$ because $r_1$ or $r_2\geq 2$.
In addition, in case 2.5 we have $r_1,s_1,s_2=1$, which means the word $w_{q-p-2}$ starts with $00101$.
In cases 2.1 and 2.2 the word $w_{q-p-3}$ starts with $001$ because $r_1,r_2\geq 2$.

We get the smallest possible $w_i$ for the rest of the words by increasing the number of zeros in front of the word by one, until $i=1$.

From Lemmas \ref{l4} and \ref{l5} we now get the following inequalities:

2.1) \begin{center}
\begin{tabular}{r @{\ } l}
$2\cdot(u_{q-p-i})_2$&$< (w_{q-p-i})_2\ \textrm{for every}\ 3\leq i\leq q-p-1$ (Lemma \ref{l4}, 1)\\
$3/2\cdot(u_{q-p-2})_2$&$< (w_{q-p-2})_2$ (Lemma \ref{l5}, 1)\\
$(u_{q-p-1})_2$&$< (w_{q-p-1})_2$ (Lemma \ref{l4}, 4)\\
$3/4\cdot(u_{q-p})_2$&$< (w_{q-p})_2$ (Lemma \ref{l5}, 6)\\
$(u_{q-p+1})_2$&$< (w_{q-p+1})_2$ (Lemma \ref{l4}, 6)\\
$2/3\cdot(u_{q-p+2})_2$&$< (w_{q-p+2})_2$ (Lemma \ref{l5}, 7)\\
$3/4\cdot(u_{q-p+i})_2$&$< (w_{q-p+i})_2\ \textrm{for every}\ 3\leq i\leq p$ (Lemma \ref{l5}, 8).\\
\end{tabular}
\end{center}

2.2)
\begin{center}
\begin{tabular}{r @{\ } l}
$8/3\cdot(u_{q-p-i})_2$&$< (w_{q-p-i})_2\ \textrm{for every}\ 3\leq i\leq q-p-1$ (Lemma \ref{l5}, 3)\\
$(u_{q-p-2})_2$&$< (w_{q-p-2})_2$ (Lemma \ref{l4}, 4)\\
$(u_{q-p-1})_2$&$< (w_{q-p-1})_2$ (Lemma \ref{l4}, 4)\\
$2/3\cdot(u_{q-p})_2$&$< (w_{q-p})_2$ (Lemma \ref{l5}, 7)\\
$(u_{q-p+1})_2$&$< (w_{q-p+1})_2$ (Lemma \ref{l4}, 6)\\
$2/3\cdot(u_{q-p+2})_2$&$< (w_{q-p+2})_2$ (Lemma \ref{l5}, 7).\\
\end{tabular}
\end{center}

2.3)
\begin{center}
\begin{tabular}{r @{\ } l}
$13/8\cdot(u_{q-p-i})_2$&$< (w_{q-p-i})_2\ \textrm{for every}\ 2\leq i\leq q-p-1$ (Lemma \ref{l5}, 4)\\
$(u_{q-p-1})_2$&$< (w_{q-p-1})_2$ (Lemma \ref{l4}, 4)\\
$3/4\cdot(u_{q-p})_2$&$< (w_{q-p})_2$ (Lemma \ref{l5}, 6)\\
$(u_{q-p+1})_2$&$< (w_{q-p+1})_2$ (Lemma \ref{l4}, 6)\\
$5/6\cdot(u_{q-p+2})_2$&$< (w_{q-p+2})_2$ (Lemma \ref{l5}, 9)\\
$3/4\cdot(u_{q-p+i})_2$&$< (w_{q-p+i})_2\ \textrm{for every}\ 3\leq i\leq p$ (Lemma \ref{l5}, 8).\\
\end{tabular}
\end{center}

2.4)
\begin{center}
\begin{tabular}{r @{\ } l}
$11/8\cdot(u_{q-p-i})_2$&$< (w_{q-p-i})_2\ \textrm{for every}\ 1\leq i\leq q-p-1$ (Lemma \ref{l5}, 2)\\
$3/4\cdot(u_{q-p})_2$&$< (w_{q-p})_2$ (Lemma \ref{l5}, 6)\\
$(u_{q-p+1})_2$&$< (w_{q-p+1})_2$ (Lemma \ref{l4}, 6)\\
$5/6\cdot(u_{q-p+2})_2$&$< (w_{q-p+2})_2$ (Lemma \ref{l5}, 9)\\
$3/4\cdot(u_{q-p+i})_2$&$< (w_{q-p+i})_2\ \textrm{for every}\ 3\leq i\leq p$ (Lemma \ref{l5}, 8).\\
\end{tabular}
\end{center}

2.5)
\begin{center}
\begin{tabular}{r @{\ } l}
$5/3\cdot(u_{q-p-i})_2$&$< (w_{q-p-i})_2\ \textrm{for every}\ 2\leq i\leq q-p-1$ (Lemma \ref{l5}, 5)\\
$(u_{q-p-1})_2$&$< (w_{q-p-1})_2$ (Lemma \ref{l4}, 4)\\
$5/6\cdot(u_{q-p})_2$&$< (w_{q-p})_2$ (Lemma \ref{l5}, 9)\\
$(u_{q-p+1})_2$&$< (w_{q-p+1})_2$ (Lemma \ref{l4}, 6)\\
$5/6\cdot(u_{q-p+2})_2$&$< (w_{q-p+2})_2$ (Lemma \ref{l5}, 9).\\
\end{tabular}
\end{center}

All we need to do anymore is to calculate that the products of the multipliers are at least one:

2.1) $2^{q-p-3}\cdot \frac{3}{2}\cdot \frac{3}{4}\cdot \frac{2}{3}\cdot \frac{3}{4}^{p-3}=\frac{9}{8}\cdot \frac{3}{2}^{p-4}2^{q-2p}>1$

2.2) $\frac{8}{3}^{q-p-3}\cdot \frac{2}{3} \cdot \frac{2}{3} = \frac{32}{27}\cdot\frac{8}{3}^{q-p-4}>1$

2.3) $\frac{13}{8}^{q-p-2}\cdot \frac{3}{4} \cdot \frac{5}{6} \cdot \frac{3}{4}^{p-3}=\frac{195}{192}\cdot \frac{39}{32}^{p-3}\frac{13}{8}^{q-2p}>1$

2.4) $\frac{11}{8}^{q-p-1}\cdot \frac{3}{4} \cdot \frac{5}{6}\cdot \frac{3}{4}^{p-3}=\frac{605}{512}\cdot \frac{33}{32}^{p-3}\frac{11}{8}^{q-2p}>1$

2.5) $\frac{5}{3}^{q-p-2} \cdot \frac{5}{6}\cdot \frac{5}{6}=\frac{125}{108}\cdot\frac{5}{3}^{q-p-3}>1.$
\end{proof}

\begin{figure}[ht]
\centering
\includegraphics[width=0.9\textwidth]{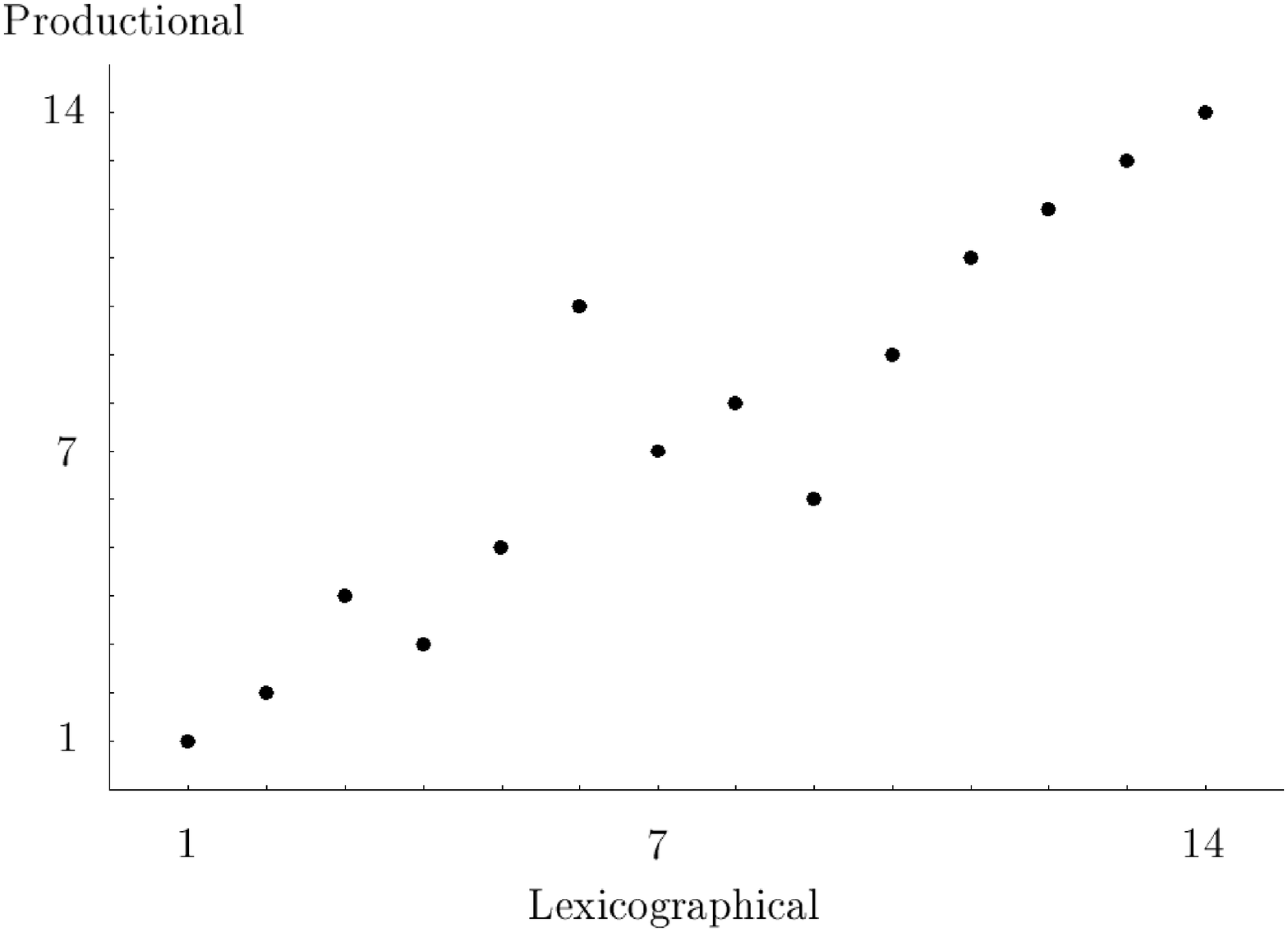}
\caption{The permutation between lexicographical and productional orders in $\mathbb{W}_{4,9}$.}
\label{pic2}
\end{figure}

We know that the lexicographically smallest orbit, the most unbalanced orbit $u$, gives the smallest product
and that the lexicographically largest orbit, the balanced orbit $b$, gives the largest product.
This does not apply generally to all the words between these two extremal words, i.e. a word may have a smaller product than a word which has smaller lexicoraphical order.

In \cite{jz} it was observed a permutation between the lexicographic ordering of an orbit $w\in\mathbb{W}_{p,q}$ and the \emph{dynamic} ordering $w,\sigma(w),\ldots,\sigma^{q-1}(w)$ of that same orbit.
They called it the \emph{lexidynamic permutation} for the word $w$.
We can also examine a permutation from the lexicographical order of an orbit in the whole $\mathbb{W}_{p,q}$ to the productional order of that orbit.
This means that the \emph{lexiproductional} permutation for the word $w\in\mathbb{W}_{p,q}$ always maps $1\mapsto1$ and $q\mapsto q$, where $|w|=q$.
Here is the permutation for $\mathbb{W}_{4,9}$ which is plotted in Figure \ref{pic2} (the product of the latter word is in parentheses):

$$000001111 \mapsto 000001111\ (17057310054912000000)$$
$$000010111 \mapsto 000010111\ (69309861547173120000)$$
$$000011011 \mapsto 000011101\ (103115999585285683200)$$
$$000011101 \mapsto 000011011\ (106107230996504524800)$$
$$000100111 \mapsto 000100111\ (184709385608811148800)$$
$$000101011 \mapsto 000111001\ (225726106934040832512)$$
$$000101101 \mapsto 000101101\ (287935726164372000000)$$
$$000110011 \mapsto 000110011\ (288046371229598615040)$$
$$000110101 \mapsto 000101011\ (294762710705942322432)$$
$$000111001 \mapsto 000110101\ (359572755909315080448)$$
$$001001011 \mapsto 001001011\ (450633542546718000000)$$
$$001001101 \mapsto 001001101\ (480928605792476688000)$$
$$001010011 \mapsto 001010011\ (524261153928446022528)$$
$$001010101 \mapsto 001010101\ (678501146123915400000)$$

We can define the balancedness of a word in $\mathbb{W}_{p,q}$ by the productional ordering, i.e. a word is more balanced than words with smaller product.
We can see that the words $000011101$ and $000111001$ have larger productional order than lexicographical order and
that the words $000011011, 000101011$ and $000110101$ have smaller productional order than lexicographical order.
For the rest of the words these orders are the same.
We can therefore define that the words $000011101$ and $000111001$ are \emph{over balanced} and that the words $000011011, 000101011$ and $000110101$ are \emph{under balanced}.
The rest of the words are \emph{equally balanced}.


\section{Partial product}

In this section we prove the counterpart for Conjecture \ref{c3}.

\begin{example}
If $(p,q)=(3,8)$ then the set of all orbits is $\mathbb{W}_{3,8}=\{00000111,00001011,$ $00001101,00010011,00010101,00011001,00100101\}$.
The base-2 orbits and the approximated partial products of those orbits are listed in Table \ref{table:tableMul}.
From these partial products we can see the partial ordering of the set $\mathbb{W}_{3,8}$ with respect to product, which is drawn in Figure \ref{pic3}.
Notice that it is different from the Figure \ref{pic1}.
\end{example}

\begin{table}[ht]
\begin{center}
\begin{tabular}{|c|c|c|c|c|c|c|c|c|c|c|c|c|c|c|}
\hline
\multicolumn{2}{|c|}{00000111} & \multicolumn{2}{|c|}{00001011}  & \multicolumn{2}{|c|}{00001101} & \multicolumn{2}{|c|}{00010011} &
\multicolumn{2}{|c|}{00010101} & \multicolumn{2}{|c|}{00011001} & \multicolumn{2}{|c|}{00100101} & $\cdot 10^{x_i}$ \\
\hline
$\mathcal{I}_i$ & $\mathcal{P}_i$ & $\mathcal{I}_i$ & $\mathcal{P}_i$ & $\mathcal{I}_i$ & $\mathcal{P}_i$ & $\mathcal{I}_i$ & $\mathcal{P}_i$ &
$\mathcal{I}_i$ & $\mathcal{P}_i$ & $\mathcal{I}_i$ & $\mathcal{P}_i$ & $\mathcal{I}_i$ & $\mathcal{P}_i$ & $x_i$\\
\hline
7 & 7 & 11 & 11 & 13 & 13 & 19 & 19 & 21 & 21 & 25 & 25 & 37 & 37 & 0\\
\hline
14 & 0.98 & 22 & 2.4 & 26 & 3.4 & 38 & 7.2 & 42 & 8.8 & 35 & 8.7 & 41 & 15 & 2\\
\hline
28 & 0.27 & 44 & $1.1$ & 52 & 1.8 & 49 & 3.5 & 69 & 6.1 & 50 & 4.4 & 73 & 11 & 4\\
\hline
56 & 0.15 & 88 & 0.94 & 67 & 1.2 & 76 & 2.7 & 81 & 4.9 & 70 & 3.1 & 74 & 8.2 & 6\\
\hline
112 & 0.17 & 97 & 0.91 & 104 & 1.2 & 98 & 2.6 & 84 & 4.1 & 100 & 3.1 & 82 & 6.7 & 8\\
\hline
131 & 0.23 & 133 & 1.2 & 134 & 1.6 & 137 & 3.6 & 138 & 5.7 & 140 & 4.3 & 146 & 9.8 & 10\\
\hline
193 & 0.44 & 176 & 2.1 & 161 & 2.6 & 152 & 5.5 & 162 & 9.3 & 145 & 6.2 & 148 & 15 & 12\\
\hline
224 & 0.97 & 194 & 4.1 & 208 & 5.5 & 196 & 11 & 168 & 16 & 200 & 12 & 164 & 24 & 14\\
\hline
\end{tabular}
\caption{The base-2 orbits and the (approximated) partial product in $\mathbb{W}_{3,8}$.}
\label{table:tableMul}
\end{center}
\end{table}

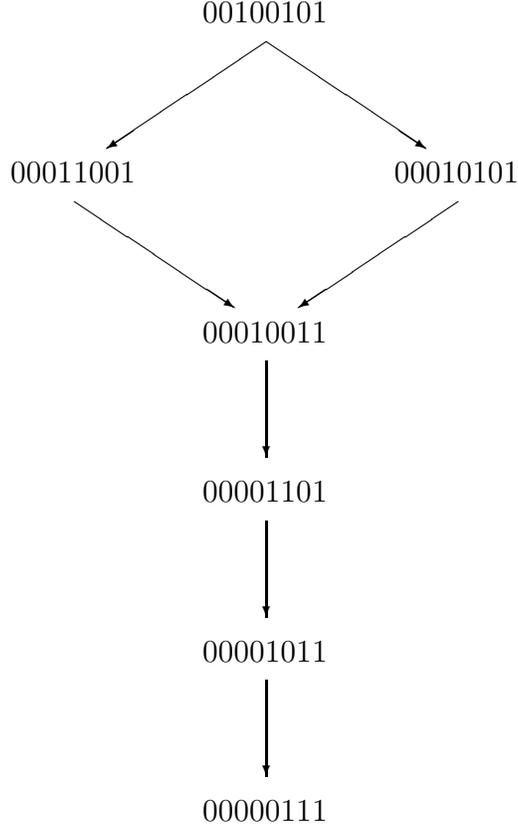
\begin{figure}[ht]
\setlength{\unitlength}{0.85mm}
\begin{picture}(150,125)
\put(85,125){$00100101$}
\put(95,122){\vector(-3,-2){25}}
\put(95,122){\vector(3,-2){25}}

\put(55,100){$00011001$}
\put(65,97){\vector(3,-2){25}}
\put(115,100){$00010101$}
\put(125,97){\vector(-3,-2){25}}

\put(85,75){$00010011$}
\put(95,72){\vector(0,-1){15}}

\put(85,50){$00001101$}
\put(95,47){\vector(0,-1){15}}

\put(85,25){$00001011$}
\put(95,22){\vector(0,-1){15}}

\put(85,0){$00000111$}
\end{picture}
\caption{The partially ordered set $(\mathbb{W}_{3,8},\prec_p)$. Similar to the partial sum, if $p$ and $q$ grow large it is hard to yield any other general results from the poset $(\mathbb{W}_{p,q},\prec_p)$ except the two extremal elements.}
\label{pic3}
\end{figure}

\begin{theorem}\label{t6}
For any integers $1\leq p < q-p$, the most unbalanced orbit $u=0^{q-p}1^p\in\mathbb{W}_{p,q}$
is the greatest element in $(\mathbb{W}_{p,q},\prec_p)$. In other words, for any $w\in\mathbb{W}_{p,q}$,
$$ \mathcal{P}_i(u) \leq \mathcal{P}_i(w)\quad {\rm for\ all}\ 1\leq i\leq q.$$
\end{theorem}
\begin{proof}
We mark $w=0^{r_1}1^{s_1}0^{r_2}1^{s_2}\ldots 0^{r_n}1^{s_n}$, where $\sum_{i=1}^{n}{s_i}=p$, $\sum_{i=1}^{n}{r_i}=q-p$, $n\geq2$ and $\forall i: r_i,s_i > 0$.
The orbits of $w$ and $u$ are marked with $(w_1,\ldots,w_q)$ and $(u_1,\ldots,u_q)$.

We use Table \ref{table:table3} from the proof of Theorem \ref{t4} and the same kind of deduction.
We get that $\mathcal{P}_i(u) \leq \mathcal{P}_i(w)$ for $1\leq i\leq q-p-1$.
If we suppose that $\mathcal{P}_{q-p}(u) \leq \mathcal{P}_{q-p}(w)$ then we get that
$\mathcal{P}_{q-p+1}(u) \leq \mathcal{P}_{q-p+1}(w)$, since $(u_{q-p+1})_2 < (w_{q-p+1})_2$.
From Theorem \ref{t5} we directly get that $\mathcal{P}_q(u) < \mathcal{P}_q(w)$.
Because $(u_i)_2 > (w_i)_2$ for $q-p+2\leq i \leq q$, we get that $\mathcal{P}_i(u) \leq \mathcal{P}_i(w)$ for $q-p+2\leq i\leq q$.

Again, the only thing we need to prove anymore is our assumption $\mathcal{P}_{q-p}(u) \leq \mathcal{P}_{q-p}(w)$ in the previous paragraph.
We divide the proof into two cases: 1) $n\geq3$ and 2) $n=2$.

\begin{table}[ht]
\begin{center}
\begin{tabular}{|l|l|l|l|l|l|c|c|c|c|c|}
\hline
\multicolumn{1}{|c|}{$i$} & \multicolumn{5}{|c|}{Prefixes of $w_i$} & \multicolumn{5}{|c|}{Multiplier(i)} \\ \cline{2-11}
& \multicolumn{1}{|c|}{1} & \multicolumn{1}{|c|}{2.1} & \multicolumn{1}{|c|}{2.2} & \multicolumn{1}{|c|}{2.3} & \multicolumn{1}{|c|}{2.4} & 1 & 2.1 & 2.2 & 2.3 & 2.4 \\
\hline
$q-p-3$ & & 001 & & & & & 2 & & & \\
\hline
$q-p-2$ & 01 & 001 & 001101 & 001011 & 001 & 2 &  & 13/8 & 11/8 & 4/3 \\
\hline
$q-p-1$ & 01 & 01 & 01 & 01011 & 01 & & & & 11/8 & 4/3 \\
\hline
$q-p$ & 01 & 01 & 011 & 011 & 01 & 1/2 & 1/2 & 3/4 & 3/4 & 2/3 \\
\hline
\end{tabular}
\caption{Prefixes and multipliers of $w_i$ from the proof of Theorem \ref{t6}.}
\label{table:table6}
\end{center}
\end{table}

1) Because $n\geq3$ the words $w_{q-p}, w_{q-p-1}$ and $w_{q-p-2}$ start with $01$.
From Lemma \ref{l4}, 1 and 5, we now directly get that $1/2\cdot(u_{q-p})_2 < (w_{q-p})_2$ and $2\cdot(u_{q-p-2})_2 < (w_{q-p-2})_2$.
This gives our claim:
$$ \mathcal{P}_{q-p}(u)=\prod_{i=1}^{q-p}{(u_{i})_2} = \prod_{i=1}^{q-p-3}{[(u_{i})_2]}\cdot 2(u_{q-p-2})_2\cdot (u_{q-p-1})_2 \cdot 1/2(u_{q-p})_2
< \prod_{i=1}^{q-p}{(w_{i})_2} = \mathcal{P}_{q-p}(w). $$

2) We divide this case into four subcases depending on the values of $r_i$ and $s_i$.
Notice that case $r_1,r_2=1$ is impossible because then we would have $2\leq s_1+s_2=p<q-p=r_1+r_2=2$.

2.1) $r_1,r_2\geq2$. Now the words $w_{q-p}$ and $w_{q-p-1}$ start with 01 and the words $w_{q-p-2}$ and $w_{q-p-3}$ start with 001.
From Lemma \ref{l4}, 1 and 5, we again get that $1/2\cdot(u_{q-p})_2 < (w_{q-p})_2$ and $2\cdot(u_{q-p-3})_2 < (w_{q-p-3})_2$.
This gives our claim:
$$ \mathcal{P}_{q-p}(u) = \prod_{i=1}^{q-p-4}{[(u_{i})_2]}\cdot 2(u_{q-p-3})_2\cdot (u_{q-p-2})_2 \cdot (u_{q-p-1})_2\cdot 1/2(u_{q-p})_2
< \prod_{i=1}^{q-p}{(w_{i})_2} = \mathcal{P}_{q-p}(w).$$

2.2) $r_1=1,r_2\geq2$ and $s_1\geq1,s_2\geq2$. This is identical to the case 2.3 in the proof of Theorem \ref{t5}, from which
we get that $w_{q-p}$ starts with $011$,$w_{q-p-1}$ starts with $01$ and $w_{q-p-2}$ starts with $001101$.
From Lemma \ref{l5}, 4 and 6, we similarly get that $3/4\cdot(u_{q-p})_2 < (w_{q-p})_2$ and $13/8\cdot(u_{q-p-2})_2 < (w_{q-p-2})_2$.
This gives our claim (notice that $13/8\cdot3/4>1$):
$$ \mathcal{P}_{q-p}(u) < \prod_{i=1}^{q-p-3}{[(u_{i})_2]}\cdot 13/8(u_{q-p-2})_2\cdot (u_{q-p-1})_2 \cdot 3/4(u_{q-p})_2
< \prod_{i=1}^{q-p}{(w_{i})_2} = \mathcal{P}_{q-p}(w).$$

2.3) $r_1=1,r_2\geq2$ and $s_1\geq2,s_2\geq1$. This is similar to the case 2.4 in the proof of Theorem \ref{t5}, from which
we get that $w_{q-p}$ starts with $011$, $w_{q-p-1}$ starts with $01011$ and $w_{q-p-2}$ starts with $001011$.
From Lemma \ref{l5}, 2 and 6, we similarly get that $3/4\cdot(u_{q-p})_2 < (w_{q-p})_2$,
$11/8\cdot(u_{q-p-1})_2 < (w_{q-p-1})_2$ and $11/8\cdot(u_{q-p-2})_2 < (w_{q-p-2})_2$.
This gives our claim (notice that $11/8\cdot11/8\cdot3/4>1$):
$$ \mathcal{P}_{q-p}(u) < \prod_{i=1}^{q-p-3}{[(u_{i})_2]}\cdot 11/8(u_{q-p-2})_2\cdot 11/8(u_{q-p-1})_2 \cdot 3/4(u_{q-p})_2
< \prod_{i=1}^{q-p}{(w_{i})_2} = \mathcal{P}_{q-p}(w).$$

2.4) $r_1=1,r_2\geq2$ and $s_1,s_2=1$. Because $r_2\geq2$ the words $w_{q-p},w_{q-p-1}$ and $w_{q-p-2}$ start with $01,01$ and $001$.
Because $q-p=s_1+s_2=2$ we have $u_{q-p}=0110^{q-3}$, $u_{q-p-1}=00110^{q-4}$ and $u_{q-p-2}=000110^{q-5}$.
Now we get that $2/3\cdot(u_{q-p})_2 = (0\frac{2}{3}\frac{2}{3}0^{q-3})_2 = (010^{q-2})_2 < (w_{q-p})_2$,
$4/3\cdot(u_{q-p-1})_2 = (00\frac{4}{3}\frac{4}{3}0^{q-4})_2 = (00\frac{6}{3}0^{q-3})_2 = (010^{q-2})_2 < (w_{q-p-1})_2$ and 
$4/3\cdot(u_{q-p-2})_2 = (000\frac{4}{3}\frac{4}{3}0^{q-5})_2 = (000\frac{6}{3}0^{q-4})_2 = (0010^{q-3})_2 < (w_{q-p-2})_2$.
This gives our claim (notice that $4/3\cdot4/3\cdot2/3>1$):
$$ \mathcal{P}_{q-p}(u) < \prod_{i=1}^{q-p-3}{[(u_{i})_2]}\cdot 4/3(u_{q-p-2})_2\cdot 4/3(u_{q-p-1})_2 \cdot 2/3(u_{q-p})_2
< \prod_{i=1}^{q-p}{(w_{i})_2} = \mathcal{P}_{q-p}(w).$$
\end{proof}


\section*{Acknowledgements}

I would like to thank my supervisor Luca Zamboni for introducing me to this problem.


\end{document}